\documentclass{amsart}

\usepackage{xcolor}
\usepackage[boxruled]{algorithm2e}
\usepackage[shortlabels]{enumitem}
\usepackage[all]{xy}
\usepackage{amssymb}
\usepackage{hyperref}

\newtheorem{theorem}{Theorem}[section]
\newtheorem{lemma}[theorem]{Lemma}
\newtheorem{prop}[theorem]{Proposition}
\newtheorem{cor}[theorem]{Corollary}
\theoremstyle{definition}

\newtheorem{example}[theorem]{Example}
\theoremstyle{remark}
\newtheorem{remark}[theorem]{Remark}
\numberwithin{equation}{section}

\def\Q{\mathbb{Q}}
\def\Z{\mathbb{Z}}
\def\C{\mathbb{C}}
\def\R{\mathbb{R}}
\def\F{\mathbb{F}}
\DeclareMathOperator{\Hom}{Hom}
\DeclareMathOperator{\Pic}{Pic}
\DeclareMathOperator{\ICM}{ICM}
\DeclareMathOperator{\rk}{Rank}
\DeclareMathOperator{\End}{End}
\DeclareMathOperator{\AV}{AV}
\DeclareMathOperator{\Tr}{Tr}
\DeclareMathOperator{\Aut}{Aut}

\newcommand{\cF}{{\mathcal F}}
\newcommand{\cI}{{\mathcal I}}
\newcommand{\cM}{{\mathcal M}}
\newcommand{\cO}{{\mathcal O}}
\newcommand{\cP}{{\mathcal P}}
\newcommand{\cQ}{{\mathcal Q}}
\newcommand{\cS}{{\mathcal S}}
\newcommand{\set}[1]{\left\lbrace#1\right\rbrace }
\newcommand{\Span}[1]{\left<#1\right>}
\newcommand{\abs}[1]{\left|#1\right|}
\newcommand{\vphi}{{\varphi}}
\newcommand{\Wpoly}[1]{{\mathcal W}({#1})}

\newcommand{\AVord}[1]{\AV^{\textit{ord}}({#1})}
\newcommand{\AVcs}[1]{\AV^{\textit{cs}}({#1})}
\newcommand{\Modord}[1]{\cM^{\textit{ord}}({#1})}
\newcommand{\Modcs}[1]{\cM^{\textit{cs}}({#1})}
\newcommand{\Mod}[1]{\cM({#1})}
\newcommand{\Fcs}{\cF^{\textit{cs}}}
\newcommand{\Ford}{\cF^{\textit{ord}}}
\newcommand{\idcl}[1]{[{#1}]}
\newcommand{\idcat}[1]{\cI({#1})}
\renewcommand{\bar}{\overline}

\newcommand{\Eq}{\cF}

\begin{document}

\begin{abstract}
We give algorithms to compute isomorphism classes of ordinary abelian varieties defined over a finite field $\F_q$ whose characteristic polynomial (of Frobenius) is square-free and of abelian varieties defined over the prime field $\F_p$ whose characteristic polynomial is square-free and does not have real roots.
In the ordinary case we are also able to compute the polarizations and the group of automorphisms (of the polarized variety) and, when the polarization is principal, the period matrix.
\end{abstract}

\author{Stefano Marseglia}
\address{Matematiska institutionen, Stockholms universitet, Sweden}
\curraddr{Mathematical Institute, Utrecht University, The Netherlands}
\email{s.marseglia@uu.nl}
\title{Computing square-free polarized abelian varieties over finite fields}


\date{} % which date should I put here ?

%\dedicatory{}

\maketitle

%\textbf{Keywords:} Abelian varieties $\cdot$ Finite Fields $\cdot$ Polarizations

\section{Introduction}

It is well known that the abelian varieties of dimension $g$ defined over the complex numbers can be functorially (and explicitly) described in terms of full lattices $L$ in $\C^g$ such that the associated complex torus $\C^g/L$ admits a Riemann form, see for example \cite{RosenAbVarC86}.

When we move to the world of positive characteristic $p$, thanks to Serre, we know that we cannot describe the whole category of abelian varieties of dimension~$g$ in terms of lattices of rank $2g$. This is due to the existence of objects such as supersingular elliptic curves whose endomorphism algebra is quaternionic and hence does not admit a $2$-dimensional representation over $\Q$.

Nevertheless, if we restrict our attention to some subcategories of the category of abelian varieties defined over a finite field we have equivalences with the category of finitely generated free $\Z$-modules with extra structure satisfying some easy-to-state axioms.
More precisely, this was proved by Deligne in \cite{Del69} for ordinary abelian varieties over a finite field $\F_q$, where $q=p^r$ is an arbitrary prime power, and by Centeleghe-Stix in \cite{CentelegheStix15} for abelian varieties over the prime field $\F_p$ whose characteristic polynomial of Frobenius does not have real roots, or equivalently, such that $\sqrt{p}$ is not an eigenvalue of the action of Frobenius on the associated $l$-adic Tate module, for any prime $l\neq p$.
Other functors (which we do not use in this paper) defined on the subcategory of powers of elliptic curves are studied in the Appendix in \cite{Lauter02}, in \cite{Kani11} and in \cite{JKPRSBT18}.

The \emph{main result} of this paper is an algorithm to compute the isomorphism classes of abelian varieties in the isogeny class determined by a \emph{square-free} characteristic polynomial $h$ of Frobenius using Deligne and Centeleghe-Stix' results.
The key point to perform this computation is that the target category of Deligne's and Centeleghe-Stix equivalences is equivalent to the category of fractional ideals of the order $\Z[F,V]$, where $F$ is a root of $h$ and $V=q/F$ in the Deligne case and $V=p/F$ in Centeleghe-Stix case.
Fractional ideals for orders that are not domains will be defined in Section \ref{sec:orders}.
The order $\Z[F,V]$ might not be maximal and so the fractional ideals might not be invertible, even in their own multiplicator ring.
In \cite{MarICM18} we describe a method to compute the isomorphism classes of such ideals and hence we are able to compute the isomorphism classes of abelian varieties in the isogeny class determined by $h$, see Algorithm \ref{alg:isoclass}.

In the ordinary case, translating the results of \cite{Howe95} into our ideal-theoretic description allows us to compute polarizations of arbitrary degree and the automorphism group of the polarized abelian varieties, see Algorithms \ref{alg:pol} and \ref{alg:autom}.

The present algorithm could be used to provide computational evidence for extending the formulas counting the number of isomorphism classes of principally polarized abelian varieties such as in \cite{AchterWilliams15} and \cite{AchterGordon17}.

We would like to stress that the shift from $\Z[F,V]$-modules to $\Z[F,V]$-fractional ideals (for certain isogeny classes) is very natural and it has already been used in the past, sometimes implicitly.
A list of papers where such a shift is applied to simple ordinary varieties includes:
the work of Howe, see for example 
\cite[Section~6]{Howe95} where the focus is on abelian varieties with maximal endomorphism ring and
\cite[Section~2]{how04} for abelian surfaces;
a paper by Lenstra, cf.~\cite[Section~6]{Lenstra96};
a paper by Lenstra, Pila and Pomerance, with focus on for abelian surfaces, cf.~\cite[Section~8]{LenstraPilaPomerance02};
a paper by Shankar and Tsimerman, mainly for geometrically simple isogeny classes, cf.~\cite[Section~3.1]{ShankarTsimerman18}.
The shift is also described using categorical language for simple ordinary abelian varieties in previous work of the author \cite{Mar16}, and in Martindale's thesis 
\cite[Sections~1.2,1.4]{MartindaleThesis}.
Results analogous to the one contained in the present paper for simple almost ordinary abelian varieties in odd characteristic can be found in Oswal and Shankar's paper \cite[Section~4]{OswalShankarEarlyView}.

The paper is structured as follows.
In Section~\ref{sec:orders} we recall the definition of fractional ideal of an order and we introduce the notion of an ideal class monoid. 
In Section~\ref{sec:categories} we describe the categories of abelian varieties and Deligne's and Centeleghe-Stix' equivalences.
In Section~\ref{sec:squarefreecase} we focus on the square-free case and prove an equivalence with the category of fractional ideals of certain orders.
%\\\textcolor{red}{gppts}\\
Such an equivalence allows us to describe the endomorphism ring, the automorphism group and the group of rational points of the abelian varieties.
In Section~\ref{sec:polarizations} we translate the notion of a polarization of an ordinary abelian variety over a finite field into the ideal-theoretic language and we describe how to compute the polarizations of a given degree up to isomorphisms. We also describe how to compute the automorphism group of the polarized abelian variety.
In Section~\ref{sec:algorithms} we present the algorithms from the previous sections and in Section \ref{sec:examples} we present the output of some computations.
Finally, in Section~\ref{sec:periodmatrices} we explain how to compute a period matrix of the canonical lift of an ordinary principally polarized abelian variety using the tools provided.
The algorithms have been implemented in Magma \cite{Magma}. 
The packages and the code to reproduce the examples contained in this paper are available at
\url{ https://github.com/stmar89/AbVarFq }.
The author is currently running a big computation of isomorphisms classes of abelian varieties over finite fields. The output will be published on \cite{LMFDB}.

\subsection*{Acknowledgments}
The author would like to thank Jonas Bergstr\"om for helpful discussions and Rachel Newton and Christophe Ritzhentaler for comments on a previous version of the paper, which is part of the author's Ph.D thesis \cite{MarPhDThesis18}.
The author would also like to express his gratitude to the Max Planck Institute for Mathematics in Bonn for their hospitality.
The author thanks the anonymous reviewer of Mathematics of Computation for useful comments and suggestions.

\subsection*{Conventions}

All rings are commutative and unital.
All morphisms between abelian varieties $A$ and $B$ over a field $k$ are also defined over $k$, unless otherwise specified.
In particular we write $\Hom(A,B)$ for $\Hom_k(A,B)$.
An abelian variety $A$ is simple if it is so over the field of definition.

\section{ Orders and Ideal classes }
\label{sec:orders}

Let $f\in \Q[x]$ be a monic square-free polynomial and denote by $K$ the \'etale $\Q$-algebra $\Q[x]/(f)$.
Note that $K$ is a finite product of number fields.
An \emph{order} in $K$ is a subring $R$ of $K$ whose underlying abelian group is isomorphic to $\Z^n$ where $n=\deg f$.
In particular, we have that $R\otimes_\Z \Q = K$ and that $K$ is the total ring of quotients of $R$.
Among all orders of $K$ there is one maximal with respect to inclusion called the \emph{maximal order of $K$}. It is a Dedekind ring and we denote it by~$\cO_K$.

A \emph{fractional $R$-ideal} is a finitely generated sub-$R$-module $I$ of $K$ such that $K=I\otimes_\Z\Q$, or, equivalently, it contains a non-zero divisor of $K$.
% Since $I$ is f.g. as an $R$-module, then there exists a $d$ (equal to the denominators of the generators) such that $dI\subseteq R$.
% If $x\in I$ is a non-zero-divisor then $(1/x)I\subseteq R$
% In particular $I$ and $R$ have the same rank as abelian groups.
Given two fractional $R$-ideals $I$ and $J$ we have that $IJ$, $I+J$, $I \cap J$ and $(I:J)=\set{x \in K : xJ\subseteq I}$ are also fractional $R$-ideals. Note that the fractional $R$-ideals are precisely the lattices in $K$ which are $R$-modules.
% It is intended that the $R$-module structure is the one induced by the multiplication in $K$.

An \emph{over-order} of $R$ is an order containing $R$.
% Equivalently, the over-orders of $R$ can be defined as the idempotent fractional $R$-ideals.

To every fractional $R$-ideal $I$
% , and actually to every lattice in $K$,
we can associate a particular order $S$, the \emph{multiplicator ring} of $I$, defined as the biggest subring of $K$ for which $I$ is an $S$-module.
Note that $S=(I:I)$ and that $S$ is an over-order of $R$.
A fractional ideal $I$ is \emph{invertible} if $I(S:I)=S$, where $S$ is its multiplicator ring.

Observe that two fractional $R$-ideals $I$ and $J$ are isomorphic as $R$-modules if and only if there exists $a\in K^\times$ such that $I=aJ$. 
Indeed, every $R$-linear morphism $\alpha\colon I\to J$ induces a unique $K$-linear endomorphism $\alpha\otimes \Q$ of $K$ which is  uniquely determined by the image of $1$, say $a$. Moreover $\alpha$ is injective if and only if $a$ is not a zero-divisor.
We will denote by $\idcat{R}$ the category of fractional $R$-ideals with $R$-linear morphisms.

The set of fractional $R$-ideals up to isomorphism is called the \emph{ideal class monoid} of $R$ and it is denoted $\ICM(R)$.
It is a multiplicative monoid under the operation induced by ideal multiplication and contains a group $\Pic(R)$ consisting of the classes of invertible $R$-ideals, with equality if and only if $R=\cO_K$. More generally, we have that 
\[\ICM(R) \supseteq \bigsqcup_S \Pic(S),\]
where the disjoint union is taken over all the over-orders $S$ of $R$.
We will write $\idcl{I}$ for the isomorphism class of the fractional $R$-ideal $I$. 

In \cite{MarICM18} we describe an algorithm that computes $\ICM(R)$.

\section{ The category of abelian varieties over a finite field }
\label{sec:categories}
Let $q=p^r$ be a power of a prime $p$. Denote with $\AV(q)$ the category of abelian varieties over $\F_q$.
For $A$ in $\AV(q)$ let $h_A$ be the characteristic polynomial of the Frobenius acting on the Tate module $T_lA$ for a prime $l \neq p$. Recall that the definition of $h_A$ does not depend on the choice of the prime $l$.
It follows from the results of Honda \cite{Honda68} and Tate \cite{Tate66} that the polynomial $h_A$ characterizes the isogeny class of $A$. Their results can be summarized as follows. Consider the following conditions for a polynomial $h$ in $\Q[x]$:
 \begin{enumerate}[(a)]
  \item \label{item1} $h$ is monic, of even degree and with integer coefficients;
  \item \label{item2} every complex root of $h$ has absolute value $ \sqrt{q} $;
  \item \label{item3} $h = m^n $, where $m$ is irreducible and $n$ is the least common denominator of the rational numbers $\set{ v_p(f(0))/r }$, where $f$ runs over the irreducible factors of $m$ over $\Q_p$ and $v_p$ is the $p$-adic valuation normalized such that  $r=v_p(q)$. If $m$ has a real root then one needs to add  $1/2$ to the set of rational numbers.
 \end{enumerate}
 Let $\Wpoly{q}$ be the set of finite products of polynomials satisfying \ref{item1}, \ref{item2} and \ref{item3}.
\begin{prop}[Honda-Tate theory, see \cite{Tate71}]
 If $A$ is an abelian variety in $\AV(q)$ then $h_A$ is in $\Wpoly{q}$.
 Conversely, for every $h$ in $\Wpoly{q}$ there exists an abelian variety $A$ in $\AV(q)$ such that $h=h_A$.
 Given two abelian varieties $A$ and $A'$ in $\AV(q)$ we have $h_A=h_{A'}$ if and only if $A$ and $A'$ are isogenous.
 Moreover, if $A$ has dimension $g$ then $h_A$ has degree $2g$.
\end{prop}

For $h$ in $\Wpoly{q}$ we will denote by $\AV(h)$ the full subcategory of $\AV(q)$ consisting of abelian varieties in the isogeny class determined by $h$.
A polynomial $h$ in $\Wpoly{q}$ will be called \emph{ordinary} if exactly half of the roots of $h$ are $p$-adic units.
An abelian variety $A$ is called ordinary if $h_A$ is ordinary, or, equivalently, if $h_A\mod p$ is not divisible by $x^{g+1}$, where $g$ is the dimension of $A$.

The main theoretical tools we will use to understand the category $\AV(q)$ are certain functors to the category of finitely generated free $\Z$-modules with some extra structure, which become equivalences when we restrict to subcategories of $\AV(q)$. More precisely, we will consider the following categories:
\begin{itemize}
 \item $\AVord{q}$: ordinary abelian varieties over $\F_q$;
 \item $\AVcs{p}$: abelian varieties $A$ over $\F_p$ such that $h_A$ has no real root.
 \item $\Modord{q}$: free finitely generated $\Z$-modules $T$ with an endomorphism $F$ such that:
       \begin{itemize}
        \item $F\otimes_\Z \Q$ acts semi-simply on $T\otimes_\Z \Q$;
        \item the characteristic polynomial $h_F$ of $F\otimes_\Z\Q$ is in $\Wpoly{q}$;
        \item $h_F$ is ordinary;
        \item there exists an endomorphism $V$ of $T$ such that $F\circ V = q$.
       \end{itemize}
 \item $\Modcs{p}$: free finitely generated $\Z$-modules $T$ with an endomorphism $F$ such that:
       \begin{itemize}
        \item $F\otimes_\Z \Q$ acts semi-simply on $T\otimes_\Z \Q$;
        \item the characteristic polynomial $h_F$ of $F\otimes_\Z\Q$ is in $\Wpoly{p}$;
        \item $h_F$ has no real roots, that is, $h_F(\pm\sqrt{p})\neq 0$;
        \item there exists an endomorphism $V$ of $T$ such that $F\circ V = p$.
       \end{itemize}       
\end{itemize}

A morphism from $(T,F)$ to $(T',F')$ for objects in $\Modord{q}$ (or in $\Modcs{p}$) is a $\Z$-linear morphism $\vphi\colon T\to T'$ such that the following diagram commutes:
\[ \xymatrix{
T \ar[d]_F \ar[r]^\vphi & T'\ar[d]^{F'}\\
T \ar[r]^\vphi          & T'} \]

\begin{theorem}
\label{thm:functorial_descriptions}
 There are equivalences of categories
 \[\Ford:\AVord{q} \to \Modord{q}\]
 and
 \[\Fcs:\AVcs{p} \to \Modcs{p},\]
 such that if
 \[ A \mapsto (T,F)\]
 then $\rk_\Z(T) = 2\dim A$ and $F$ corresponds to the Frobenius endomorphism of $A$.
\end{theorem}
\begin{proof}
 See \cite[Theorem 7]{Del69} and the covariant version of \cite[Theorem 1]{CentelegheStix15} given in \cite[7.4]{CentelegheStix15}.
\end{proof}

\begin{remark}
\label{rmk:Delfun}
 Let $A$ be in $\AVord{q}$. We will recall the construction of $\Ford(A)=(T,F)$ given in \cite{Del69} since it will be used later in computing the polarizations.
 Denote by $W$ the ring of Witt vectors over $\F_q$.
 Since $A$ is ordinary it admits a \emph{canonical lift} to an abelian variety $\tilde A$ over $W$, characterized by $\End_{\bar\F_q}(A) = \End_{W}(\tilde A)$.
 Fix an embedding $\varepsilon\colon W\hookrightarrow \C$ and define $A'= \tilde A \otimes_\varepsilon \C$.
 Finally put $T=H_1(A',\Z)$. Note that this construction is functorial in $A$ and in particular $T$ comes equipped with an endomorphism $F$ corresponding to the Frobenius of $A$.
\end{remark}
\begin{remark}
The construction of the functor $\Fcs$ depend also on a choice, see \cite[Section~7.3]{CentelegheStix15}.
For any choice of embedding $\varepsilon\colon W\hookrightarrow \C$ the functor $\Fcs$ can be constructed in a way that extends $\Ford$ on $\AVord{p}$.
See  \cite[Proposition~45]{CentelegheStix15}.
\end{remark}

As Serre has pointed out, functorial descriptions such as the ones in Theorem \ref{thm:functorial_descriptions} cannot be extended to the whole category of abelian varieties. This is a consequence of the existence of objects like supersingular elliptic curves, whose endomorphism algebra is a quaternionic algebra which does not admit a $2$-dimensional representation.

\section{ The square-free case }
\label{sec:squarefreecase}
In this section $h$ will be either a \emph{square-free} ordinary polynomial in $\Wpoly{q}$ or a \emph{square-free} polynomial in $\Wpoly{p}$ with no real roots.
We will denote with $\Mod{h}$ the image of $\AV(h)$ under the functor $\Ford$ (or $\Fcs$, respectively).

\begin{remark}
 Let $A$ an abelian variety in $\AV(h)$. The Poincar\'e reducibility theorem states that there are simple and pairwise non-isogenous abelian varieties $B_1,
 \ldots, B_r$ and an isogeny such that 
 \[ A \sim B_1\times \ldots \times B_r.\]
 In particular $h=\prod_i h_{B_i}$ and, since $h$ is square-free, it follows that each $h_{B_i}$ is irreducible. 
 Observe that the converse holds in both cases of interest to us: the characteristic polynomial of a simple abelian variety $B$ is irreducible, hence equal to the minimal polynomial of the Frobenius, if $B$ is in $\AVord{q}$, see \cite[Theorem~3.3]{Howe95}, or in $\AVcs{p}$, because in the condition~\ref{item3} stated at the beginning of Section~\ref{sec:categories} all denominators are equal to $1$.
 
 One observes that the proportion of square-free polynomials among non-ordinary $p$-Weil polynomials is smaller than the proportion of square-free polynomials among ordinary $p$-Weil.
 Nevertheless it accounts for the vast majority of them. For example by looking at \cite{LMFDB} one sees that among the $105600$ ordinary isogeny classes of abelian fourfolds over $\F_5$ exactly $104746$ are square-free. Among the $27239$ non-ordinary isogeny classes of abelian fourfolds over $\F_5$ we have $26765$ which are square-free.
\end{remark}
\begin{remark}
 Note that being square-free is not a geometric condition, in the sense that in general it is not stable under extensions of the base field.
 For example, if $A$ is an abelian surface over $\F_{31}$ with characteristic polynomial
 \[ h_A= (x^2 - 3x + 31)(x^2 + 3x + 31) \]
 then $A$ is isogenous to the product of two non-isogenous elliptic curves $E_1$ and~$E_2$. On the other hand $E_1$ and $E_2$ become isogenous over $\F_{31^2}$ and indeed the characteristic polynomial of $A':=A\otimes \F_{31^2}$ is
 \[ h_{A'}=(x^2 + 53x + 961)^2.\]
\end{remark}

Denote with $K$ the \'etale algebra $\Q[x]/(h)$. Put $\alpha = x \mod (h)$.
Let $R$ be the order in $K$ generated by $\alpha$ and $q/\alpha$.
Observe that our order $R$ is the order $R_w$ defined in \cite[Section 2]{CentelegheStix15} for the Weil support $w$ identified by the polynomial $h$ and, similarly, $R$ equals the order $R_\C$ defined in \cite{Howe95}. 
Recall that $\idcat{R}$ denotes the category of fractional $R$-ideals.

\begin{theorem}
\label{thm:isomclass}
 There is an equivalence of categories $\Psi\colon\Mod{h}\to\idcat{R}$.
\end{theorem}
\begin{proof}
  Let $g$ be the dimension of any abelian variety in $\AV(h)$, or equivalently let~$2g$ be the degree of $h$. Pick an object $(T,F)$ in $\Mod{h}$.
  Note that by definition~$T$ is a $\Z[F,V]$-module.
  Since $h$ is square-free it is the minimal polynomial of $F$ and hence the morphism $F\mapsto \alpha$ induces an isomorphism $\Z[F,V]\simeq R$ and hence an $R$-module structure on $T$.
  Since $T$ is torsion-free it can be embedded in $R\otimes_\Z \Q = K$ and hence it can be identified with a sub-$R$-module $I$ of $K$.
  Since $T$ is an abelian group of rank $2g$ it follows that $I$ is a fractional $R$-ideal, hence an object of $\idcat{R}$.
  Denote this association $(T,F)\mapsto I$ by $\Psi$.
  Observe that $\Psi$ is a functor.
  Indeed if $\vphi\colon(T,F)\to (T',F')$ is a morphism in $\Mod{h}$, then the compatibility rule $\vphi\circ F=F' \circ \vphi$ implies that $\Psi(\vphi)$ will be an $R$-linear morphism, as required, and that it respects composition and that it sends the identity morphism to the identity morphism.
  By construction it is clear that $\Psi$ is fully faithful and essentially surjective, hence an equivalence of categories.
\end{proof}

\begin{cor}
\label{cor:isomclass}
 If $h$ is ordinary or if $h$ is over $\F_p$ with no real roots then there is an equivalence of categories
 \[ \Eq\colon\AV(h)\to \idcat{R}. \]
 In particular, $\Eq$ induces a bijection
 \[ \dfrac{\AV(h)}{ \simeq } \longrightarrow \ICM(R). \]
\end{cor}
\begin{proof}
 The functor $\Eq$ is the composition of the functor $\Ford$ (or $\Fcs$) from Theorem~\ref{thm:functorial_descriptions} together with the functor $\Psi$ from Theorem~\ref{thm:isomclass}, which are all equivalences.
\end{proof}

\begin{remark}
  Theorem \ref{thm:isomclass} and Corollary~\ref{cor:isomclass} tell us that the abelian varieties in the isogeny class $\AV(h)$ correspond to the different $\Z[x,y]/(h(x),xy-q)$-structures that one can put on $\Z^{2g}$. 
\end{remark}

\begin{cor}
\label{cor:prod}
 If $\Eq(A)=I$, then 
 \begin{enumerate}[(a)]
    \item \label{a:cor:prod} $\Eq(\End(A))=(I:I)$;
    \item \label{b:cor:prod} $\Eq(\Aut(A))=(I:I)^\times$;
    \item \label{c:cor:prod} $A$ is isomorphic to a product of abelian varieties if and only if $(I:I)$ is a product of orders.
 \end{enumerate}
\end{cor}
\begin{proof}
   Observe that \ref{a:cor:prod} and \ref{b:cor:prod} follow immediately from the previous proposition. Statement \ref{c:cor:prod} holds by functoriality and the fact that an ideal $I$ admits a decomposition $I_1\oplus I_2$ if and only if the same holds for its multiplicator ring. Indeed, let $S=(I:I)$. If $S=S_1\oplus S_2$, denote with $e_1$ and $e_2$ the units of $S_1$ and $S_2$, respectively, then $I=I_1\oplus I_2$ where $I_i=e_iI$.
   The other implication follows from the fact that if $I=I_1\oplus I_2$ then $(I:I)=(I_1:I_1)\oplus(I_2:I_2)$.
\end{proof}

%\textcolor{red}{gppts}\\
Recall that for a fractional $R$-ideal $J$ the trace dual is defined as
\[J^t=\set{ z \in K : \Tr_{K/\Q}(zJ)\subseteq \Z},\]
which is also a fractional $R$-ideal, with the same multiplicator ring as that of $J$.
Moreover, if $J=\alpha_1\Z\oplus \ldots \oplus \alpha_n \Z$, then
$J^t = \alpha_1^*\Z\oplus \ldots \oplus \alpha_n^* \Z$, where the $\alpha^*_j$'s are uniquely defined by the relations $\Tr_{K/\Q}(\alpha_i\alpha_j^*)=1$ if $i=j$ and $0$ otherwise.
\begin{cor}
\label{cor:grp_pts}
If $\Eq(A)=I$, then there is an isomorphism
\[  A(\F_q) \simeq \frac{I}{(1-F)I}. \]
\end{cor}
\begin{proof}
Observe that $A(\F_q)$ is the kernel of $1-\pi_A$, where $\pi_A$ is the Frobenius endomorphism of $A$.
Moreover, notice that the action of $F$ on $I/(1-F)I$ is invertible.
Now, if $A$ is ordinary, the statement follows from
\cite[Lemma~4.13, Proposition~4.14]{Howe95}.
% \cite[Section 4]{Del69}.

If $A \in \AVcs{p}$ we will obtain the result by looking at the $\ell$-primary parts for every prime $\ell$ diving the number of $\F_p$-points of $A$.
For a fractional $R$-ideal $J$ put $J_\ell=J\otimes_\Z \Z_\ell$.
Also denote by $N_\ell$ the $\ell$-primary part of $A(\F_p)$.

By \cite[Propositions~21 and~28]{CentelegheStix15} we have $T_p(A) \simeq I \otimes R_p \simeq I_p$ and hence
\[ N_p \simeq \frac{I_p}{(1-F)I_p}. \]

For a prime $\ell \neq p$ by \cite[Propositions~21 and~27]{CentelegheStix15} we have an isomorphism
$ T_\ell(A) \simeq \Hom_{R_\ell}(I_\ell,R_\ell) $ and hence
\[ N_\ell \simeq \frac{(R_\ell:I_\ell)}{(1-F)(R_\ell:I_\ell)}. \]
%\textcolor{red}{I need to move the definition of trace dual before this proof}\\
Recall that $R_\ell$ is Gorenstein if and only if $R^t_\ell$ is principal.
Hence we have that $(R_\ell:I_\ell) =(R^t_\ell I)^t \simeq I_\ell^t$.
Also, recall that for fractional ideals $J_1$ and $J_2$ we have an isomorphism of abelian groups $J_1^t/J_2^t\simeq J_2/J_1$.
 Hence we obtain isomorphisms of abelian groups
\[ \frac{(R_\ell:I_\ell)}{(1-F)(R_\ell:I_\ell)} \simeq \frac{I_\ell^t}{(1-F)I_\ell^t} \simeq
\frac{\frac{1}{(1-F)}I_\ell}{I_\ell} \simeq
 \frac{I_\ell}{(1-F)I_\ell}, \]
which concludes the proof at $\ell\neq p$.
\end{proof}

\section{Polarizations and automorphisms in $\AVord{q}$}
\label{sec:polarizations}

Let $h$ be an ordinary square-free polynomial in $\Wpoly{q}$ and define $K$ and $R$ as above.
Observe that $K$ is a CM-algebra with involution defined by $\overline{\alpha}=q/\alpha$.
Note that $\bar{R}=R$.
\begin{lemma}
Let $I$ be a fractional $R$-ideal. Then
\[(\bar{I})^t = \bar{(I^t)}. \]
\end{lemma}
\begin{proof}
 For $z\in K$, let $m_z$ be its minimal polynomial over $\Q$.
 Observe that $m_{\bar{z}} = m_{z}$ and in particular $\Tr_{K/\Q}(z) = \Tr_{K/\Q}(\bar{z})$.
 It follows that
 \begin{equation*}
 \begin{split}
 a\in (\bar{I})^t & \Longleftrightarrow \Tr_{K/\Q}(a\bar{i})\in \Z \text{ for every }i\in I  
 \Longleftrightarrow \\
		  & \Longleftrightarrow \Tr_{K/\Q}(\bar{a}i)\in \Z \text{ for every }i\in I \Longleftrightarrow a\in \bar{(I^t)},
 \end{split}   
 \end{equation*}
 which concludes the proof.
\end{proof}
In this section we describe how to compute the dual abelian variety, polarizations and automorphisms of a polarized abelian variety in $\AV(h)$.

\begin{theorem}
\label{thm:dual}
   Let $A$ be an abelian variety in $\AV(h)$ and $I=\Eq(A)$ be the corresponding ideal in $\idcat{R}$, where $\Eq$ is the functor of Corollary \ref{cor:isomclass}.
   Then $\bar{I}^t=\Eq(A^\vee)$, where $A^\vee$ denotes the dual abelian variety of $A$.
   Moreover, if 
   \[\Eq(\lambda\colon A\to B) = I \overset{\dot a}{\to} J\]
   then
   \[\Eq(\lambda^\vee \colon B^\vee \to A^\vee)=\bar J^t \overset{\dot {\bar a}}{\to} \bar I^t,\]
   where $\lambda^\vee$ is the morphism dual to $\lambda$ and $\dot a$ (resp.~$\dot {\bar a}$) denotes the $R$-linear morphism multiplication-by-$a$ (resp.~multiplication-by-$\bar a$).
\end{theorem}
\begin{proof}
 Let $(T,F)=(T_A,F_A)$ be the module in $\Mod{h}$ corresponding to $A$. By~\cite[Proposition~4.5]{Howe95} the dual abelian variety $A^\vee$ corresponds to 
 \[ (T_{A^\vee},F_{A^\vee})=(T^\vee,F^\vee),\]
 where $T^\vee = \Hom_\Z(T,\Z)$ and $F^\vee(\psi) = \psi \circ V$, for every $\psi\in T^\vee$. Let $n$ be the degree of $h$.
 Fix a $\Z$-basis $\alpha_1,\ldots,\alpha_n$ of $I$ and consider the $\Z$-linear maps:
 \begin{align*}
  &\Hom_\Z(I,\Z)\longrightarrow \bar{I}^t & &\text{and} & & \bar{I}^t\longrightarrow \Hom_\Z(I,\Z)\\
  &\psi\longmapsto\sum_{i=1}^n \psi(\alpha_i)\bar{\alpha_i^*} & & & & z\longmapsto \Tr_{K/\Q}(\bar{z}\cdot -)
 \end{align*}
 These maps are clearly inverses of each other and hence 
 we have that 
 \[\Psi((T^\vee,F^\vee))=\bar{I}^t,\]
 where $\Psi$ is the functor defined in the proof of Theorem~\ref{thm:isomclass}, or  equivalently, that $\Eq(A^\vee)=\bar I^t$.
 The second statement follows in an analogous manner.
\end{proof}

A morphism $\lambda: (T,F)\to (T',F')$ in $\Modord{q}$ corresponds to an isogeny if the induced linear map $\lambda \otimes \Q$ is invertible, see \cite[Section 4, p.2368]{Howe95}.
In particular, if $\lambda:A\to B$ is a morphism in $\AV(h)$ then it is an isogeny if and only if $\Eq(\lambda)=a$ is not a zero-divisor, that is $a\in K^\times$.

An isogeny $\lambda:(T,F) \to (T^\vee,F^\vee)$ defines a bilinear map $b:T\times T \to \Z$ by $b(s,t)=\lambda(t)(s)$. For such $b$, by \cite[Theorem 1.7.4.1,p.44]{Knus91}, there exists a unique $R$-sesquilinear form $S$ on $T\otimes \Q$ such that $b=\Tr_{K/\Q} \circ S$.

Since $K$ is a CM-algebra, homomorphisms $K\to \C$ come in conjugate pairs.
A \emph{CM-type} of $K$ is a choice of pairwise non-conjugate morphisms $\vphi_1,\ldots,\vphi_g : K \to \C$, where $2g=\dim_\Q K$.
Consider the set
\begin{equation}
  \Phi:=\set{ \vphi:K\to \C : v_p(\vphi(F))>0 },
  \label{eq:cmtype}
\end{equation}
where $v_p$ is the $p$-adic valuation induced by the embedding $\varepsilon\colon W\to \C$ as in Remark~\ref{rmk:Delfun}. Note that $\Phi$ is a CM-type of $K$ since the polynomial $h$ is ordinary.

An element $a\in K$ is called \emph{totally imaginary} if $a=-\bar a$, or equivalently, if $\psi(a)$ is totally imaginary for every $\psi:K\to \C$. Such an element is said to be $\Phi$-positive (resp.~non-positive) if $\Im(\vphi(a))>0$ (resp.~$\Im(\vphi(a)) \leq 0$) for every $\vphi$ in $\Phi$.

\begin{prop}[{\cite[Proposition 4.9]{Howe95}}]
\label{prop:howepol}
An isogeny $\lambda:(T,F)\to (T^\vee,F^\vee)$ corresponds to a polarization if and only if
\begin{itemize}
 \item $S$ is a skew-Hermitian form, that is $S(t_1,t_2)=-\bar{S(t_2,t_1)}$ for every $t_1,t_2 \in T\otimes \Q$, and
 \item $S(t,t)$ is $\Phi$-non-positive for every $t\in T\otimes \Q$.
\end{itemize}
\end{prop}

\begin{theorem}
\label{thm:idpol}
 Let $h$ be a square-free ordinary polynomial in $\Wpoly{q}$ and let $A$ be an abelian variety in $\AV(h)$. Define $R$ and $K$ as above and put $I=\Eq(A)$. Then:
 \begin{enumerate}[(a)]
  \item \label{thm:idpol:a} given an isogeny $\lambda\colon A \to A^\vee$ put $a=\Eq(\lambda)$. Then $\lambda$ is a polarization if and only if $a$ satisfies:
  \begin{itemize}
   \item $aI\subseteq \bar{I}^t$,
   \item $a$ is totally imaginary, and
   \item $a$ is $\Phi$-positive.
  \end{itemize}
  Moreover, we have $\deg \lambda = [ \bar{I}^t : aI ]$.
  \item \label{thm:idpol:b} given two polarizations $\lambda$ and $\lambda'$ of $A$, there is an isomorphism $(A,\lambda)\simeq (A,\lambda')$ if and only if there exists $v\in (I:I)^\times$ such that
  \[ a = \bar{v} a' v, \]
  where $a=\Eq(\lambda)$ and $a'=\Eq(\lambda')$.
  In particular, we have
  \[ \Aut((A,\lambda)) = (I:I)^\times \cap \mu_K, \]
  where $\mu_K$ is the group of torsion units of $K$.
 \end{enumerate}
\end{theorem}
\begin{proof}
\begin{enumerate}[(a)]
 \item Let $T$ be the module associated to $A$. Let $b\colon T \times T \to \Z$ and $S\colon T_\Q\times T_\Q \to K$ be the forms associated to the polarization $\lambda$. We will use the same letters to denote the forms induced after applying the functor $\Eq$.
 Using Theorem \ref{thm:dual} we see that
 \[ b(s,t)=\Tr( \bar{at}s ), \]
 which implies that 
 \[ S(s,t)= \bar{at}s. \]
 So by Proposition \ref{prop:howepol} we have that $a$ corresponds to a polarization if and only if $a=-\bar{a}$ and $\bar{a}$ is $\Phi$-non-positive.
%  Since $a$ is totally imaginary and in $K^\times$ we have that the last condition is equivalent for $a$ to be $\Phi$-positive.
 For the statement about the degree, see \cite[Section 4]{how04}.
 \item The element $v$ must be an automorphism of $I$, hence must be a unit of the multiplicator ring $(I:I)$.
 The diagram
 \[
  \xymatrix{
  I \ar[d]_v \ar[r]^a 	& \bar{I}^t \\
  I \ar[r]^{a'} 	& \bar{I}^t \ar[u]_{\bar{v}}
  }
 \]
 must commute, i.e.~ $a=\bar v a'v$.
 In particular, if $a'=a$ then, $a$ being a non-zero divisor (since it corresponds to an isogeny), this is equivalent to~$v\bar{v}=1$, that is $v$ is a torsion unit, see~\cite[Proposition~7.1]{Neukirch99}.
\end{enumerate}
\end{proof}
Note that such an ideal theoretic description was already used in 
\cite[Section~4]{how04} for simple abelian surfaces.

Recall that given an abelian variety $A$, a polarization $a$ is said to be \emph{decomposable} if there exist two polarized abelian varieties $(B_1,b_1)$ and $(B_2,b_2)$ and an isomorphism $\psi\colon A\to B_1\times B_2$ such that $a=\psi^\vee\circ (b_1\times b_2) \circ \psi$.
\begin{cor}
\label{cor:jacobians1}
   Let $A$ be an abelian variety in $\AV(h)$ and put $I=\Eq(A)$. Assume that $A$ admits a principal polarization $\lambda$.
   Then $(I:I)$ is a product of orders if and only if $(A,\lambda)$ is decomposable and hence it is not (geometrically) isomorphic to the Jacobian of a curve.
   In particular, if $R$ is a product of orders, then the isogeny class associated to $h$ does not contain a Jacobian.
\end{cor}
\begin{proof}
   Put $S=(I:I)$.
   Then by Corollary \ref{cor:prod} $S$ is a product if and only if every abelian variety with endomorphism ring $S$ is isomorphic to a product of abelian varieties. In particular, if any one of them admits a principal polarization, this would be decomposable by Theorem \ref{thm:idpol} 
   and hence cannot be isomorphic (as a polarized abelian variety) to the Jacobian of a curve.
\end{proof}
\begin{remark}
We cannot state a result analogous to Corollary \ref{cor:jacobians1} without assuming that $h$ is squarefree.
\end{remark}

\section{ Algorithms }
\label{sec:algorithms}
The algorithms in this section have been implemented in Magma \cite{Magma} and the code is abailable on the author's webpage.
We will use without mentioning a lot of algorithms for abelian groups, which can all be found in \cite[Section 2.4]{cohen93}.

\begin{algorithm}
% \SetAlgoLined
 \KwIn{$h$ a square-free ordinary polynomial in $\Wpoly{q}$ or a square-free polynomial in $\Wpoly{p}$ with no real roots;}
 \KwOut{a list of fractional ideals representing the isomorphism classes of the abelian varieties in the isogeny class determined by $h$;}
%  initialization\;
 $K:=\Q[x]/(h)$\;
 $F:= x \mod (h)$\;
 $V:=qF^{-1}$\;
 $R:=\Z[F,V]$\;
 \KwRet{ $\ICM(R)$ }\;
 \caption{\label{alg:isoclass} Isomorphism classes in a given isogeny class}
\end{algorithm}
\begin{theorem}
 Algorithm \ref{alg:isoclass} is correct.
\end{theorem}
\begin{proof}
 The correctness follows from Theorem \ref{thm:isomclass}.
\end{proof}
\begin{remark}
 In \cite{MarICM18} we describe in detail how to compute $\ICM(R)$ for any order~$R$ in a finite product of number fields $K$.
\end{remark}

\begin{algorithm}[H]
% \SetAlgoLined
  \KwIn{$h$ a square-free ordinary polynomial in $\Wpoly{q}$;}
  \KwOut{a CM-type $\Phi$ as in (\ref{eq:cmtype});}
%  initialization\;
  $\text{write }h=\prod_{i=1}^{r} h_i\text{ with }h_i\text{ irreducible}$\;
  $\Q(F):=\Q[x]/(h)$\;
  $M:=\text{SplittingField}(h)$\;
  $\mathfrak{P}:=\text{a maximal ideal of $M$ above $p$}$\;
  $\psi_0:=\text{a homomorphism $M\to \C$}$\;
  \For{$i=1\dots r$}{
      $d_i:=\deg(h_i)$\;
      $\Q(F_i):=\Q[x]/(h_i)$\;
      let $F_{i,1},\ldots,F_{i,{d_i}}$ be the conjugates of $F_i$ in $M$\;
  }
  $\Phi:=\{\ \}$\;
  \For{$\vphi \in \Hom(\Q(F),\C)$}{
	\If{  $\vphi(F)=(\psi_0(F_{1,{j_1}})\times\ldots\times\psi_0(F_{r,{j_r}}))$ with $F_{1,{j_1}},\ldots,F_{r,{j_r}}\in \mathfrak{P}$}{
		add $\vphi$ to $\Phi$\;
	}
  }
  \KwRet{ $\Phi$ }\;
  \caption{\label{alg:cmtype} CM-type}
\end{algorithm}

\begin{theorem}
 Algorithm \ref{alg:cmtype} is correct.
\end{theorem}
\begin{proof}
 Use the notation as in the Algorithm.
 Fixing an embedding of $\varepsilon\colon W\to \C$ as in Remark \ref{rmk:Delfun} encompasses fixing prime above $p$ and an embedding for each extension containing the fields $\Q(F_i)$, $i=1\dots r$ in a compatible way.
 Since we need a field containing all the conjugates of $F_i$ for all $i$, the most efficient choice is to work with the compositum $M$ of the Galois closures of the fields $\Q(F_i)$, which is precisely the splitting field of the polynomial $h$.
 Under these identifications, we get
 \begin{align*} 
 \vphi \in \Phi & \iff  v_p( \vphi|_{\Q(F_i)}(F_i) )>0 \text{ for } i=1,\ldots,r  \\
                & \iff \psi_0^{-1}(\vphi|_{\Q(F_i)}(F_i)) \in \mathfrak{P} \text{ for } i=1,\ldots,r.
 \end{align*}
 Since the polynomial $h$ is ordinary, the set $\Phi$ consists of exactly half of the homomorphisms $K\to \C$, one for each conjugate pair.
\end{proof}
% \begin{remark}
%   Observe that there are well known algorithms to compute the splitting field of a rational polynomial, see for example \cite[Section 6.3]{cohen93}.
% \end{remark}

\begin{algorithm}
Let $h$ be a square-free ordinary polynomial in $\Wpoly{q}$\;
Put $K:=\Q[x]/(h)$, $F:= x \mod (h)$, $V:=qF^{-1}$ and $R:=\Z[F,V]$\;
 \KwIn{a fractional $R$-ideal $I$ corresponding to an abelian variety $A$; a positive integer $N$;}
 \KwOut{a sequence $\cP$ of elements of $K^\times$ corresponding to all pairwise non-isomorphic polarizations of $A$ of degree $N$;}
%  initialization\;
 Compute the CM-type $\Phi$ using Algorithm \ref{alg:cmtype}\;
 $S:=(I:I)$\;
 $\mathcal{K}:= \Span{v\bar v : v \in S^\times}$ 
 \tcp*[r]{consider $S^\times$ and $\mathcal{K} $ as subgroups of $ (S\overline{S})^\times$}
 $Q:=S^\times/\mathcal{K} \cap S^\times$\;
 $\cQ:=\set{ \text{representatives in $S^\times$ of the elements of $Q$} }$\;
 $\cS':=\set{ \text{subgroups $H$ of $\bar{I}^t$ such that }[\bar{I}^t:H]=N } $\;
 $\cS:=\set{ H \in \cS' : H \text{ is an $R$-module with multiplicator ring }S }$\;
 $\cP:=\{\ \}$\;
 \For{$H\in \cS$}{
      \If{$(H:I)=x_0S$}{
      $\cP_H:=\{\ \}$\;
	  \For{$u \in \cQ$}{
	    $y:=x_0u$\;
	    \If{$\bar y = -y$ and $y$ is $\Phi$-positive}{
	    Append $y$ to $\cP_H$\;
	    }
	  }      
      }
 }
 $\cP':= \bigcup_{H \in \cS} \cP_H$\;
 \For{$\lambda\in \cP'$}{
      \If{there is no $\lambda' \in \cP$ such that $\lambda/\lambda' \in \mathcal{K}$ }{
	    Append $\lambda$ to $\cP$\;
	  }     
 }
 \KwRet{$\cP$}\;
 \caption{\label{alg:pol} Polarizations of a given abelian variety}
\end{algorithm}
\begin{theorem}
 Algorithm \ref{alg:pol} is correct.
\end{theorem}
\begin{proof}
% The correctness of Algorithm \ref{alg:pol} follows from part \ref{thm:idpol:a} of Theorem \ref{thm:idpol}. In order to prove that the algorithm terminates, it suffices to show that the quotient $Q$ is finite. 
%  Write $K=\prod_i K_i$, with $K_i$ a number field.
% By the Dirichlet Unit Theorem we have that $S^\times$ is a finitely generated group. which can be written as
% \[ S^\times = T \times \prod_j \Span{\zeta_j}\]
% where $T$ consists of torsion units and the $\zeta_j$'s have infinite order. In particular $T$ is finite.
% For each unit $u\in S^\times$ and homomorphism $\vphi: K \to \C$ we have that $\vphi(u/\bar u)$ lies on the unit circle.
% Hence by \cite[Proposition 7.1]{Neukirch99} the quotient $u/\bar u$ is a torsion unit of $(S\bar{S})^*$.
% In particular $\bar\zeta_j=\mu \zeta_j$ for some torsion unit $\mu$.
% Since the subgroup $\Span{v\bar v : v \in S^\times }$ is generated by elements $\zeta_j\bar\zeta_j$ and $\mu\bar\mu$, with $\mu$ a torsion unit, we get that the quotient $Q$ is finite.

 For each unit $u\in S^\times$ and homomorphism $\vphi: K \to \C$ we have that $\vphi(u/\bar u)$ lies on the unit circle.
 Hence by \cite[Proposition 7.1]{Neukirch99} the quotient $\zeta=u/\bar u$ has finite multiplicative order, say $n$.
 Then $u^{2n} = u^n(u\zeta)^n = u^n\bar{u}^n = (u\bar{u})^n $.
 In particular the abelian group $Q$ is torsion. 
 By the Dirichlet Unit Theorem the unit group $S^\times$ is a finitely generated abelian group, and therefore it follows that $Q$ is finite.
 Observe that given two fractional $R$-ideals $H$ and $I$, they are isomorphic if and only if they have the same multiplicator ring and $(H:I)$ is a principal ideal, see \cite[Proposition 4.1.(c), Corollary 4.5]{MarICM18}.
 Note that once we know that $(H:I)$ is invertible in $S$, checking whether it is a principal ideal it is a finite problem and can be done efficiently if have already computed $\Pic(S)$, see for example~\cite[6.5.5]{cohen93}.
 For each $H \in \cS$ with $x_0I=H$, the set $\cP_H$ contains all polarizations of~$I$ with image $H$ up to isomorphism by Theorem~\ref{thm:idpol}.
 So in particular the set $\cP'$ contains all polarizations of $I$ of degree $N$.
 By Part \ref{thm:idpol:b} of Theorem \ref{thm:idpol} $\lambda$ and $\lambda'$ in $\cP'$ are isomorphic if and only if $\lambda/\lambda'$ is in $\mathcal{K}$.
 This concludes the proof of the correctness of Algorithm~\ref{alg:pol}.
\end{proof}
\begin{remark}
Observe that Algorithm~\ref{alg:pol} can be simplified if $S^\times = \bar{S^\times}$. 
In this is the case, the set $\cP$ coincides with $\cP'$, that is, we can skip the last loop.
Indeed assume that $\lambda$ and $\lambda'$ are in $\cP'$, and satisfy $\lambda I=H$ and  $\lambda' I=H'$. 
If $\lambda = v\bar{v}\lambda'$ for some $v\in S^\times = \bar{S^\times}$ then $H = v\bar{v}H' = H'$, and hence $\lambda$ and $\lambda'$ are both in $\cP_H$ which implies that they must coincide, since $\cP_H$ contains only one representative for each isomorphism class.
This observation is particularly useful when we compute principal polarizations, because if $I$ admits a principally polarization then $S = \bar{S}$.
\end{remark}
\begin{algorithm}
 \KwIn{a pair $(I,x)$ corresponding to a polarized abelian variety $(A,\mu)$;}
 \KwOut{a finite abelian group $H$ corresponding to $\Aut((A,\mu))$;}
 $S:=(I:I)$\;
 $H:=\text{torsion}(S^\times)$\;
 \KwRet{$H$}\;
 \caption{\label{alg:autom} Automorphism of a polarized abelian variety}
\end{algorithm}
\begin{theorem}
 Algorithm \ref{alg:autom} is correct.
\end{theorem}
\begin{proof}
 It follows from Part \ref{thm:idpol:b} of Theorem \ref{thm:idpol}.
\end{proof}

\section{ Examples }
\label{sec:examples}
\subsection*{Elliptic curves}
\allowdisplaybreaks
Every elliptic curve $E$ comes with a unique principal polarization. This means that counting the isomorphism classes of elliptic curves over $\F_q$ is the same as counting the principally polarized ones.
The characteristic polynomial of the Frobenius endomorphism of an elliptic curve over $\F_q$ has the form $x^2+\beta x + q$ with $\abs{\beta}\leq 2\sqrt{q}$ by Hasse's Theorem.
Not every $\beta$ in this range gives rise to an isogeny class of an elliptic curve.
See \cite[Theorem 4.1]{Wat69} for a complete list.
% An elliptic curve that is not ordinary, that is, when $p$ divides $\beta$, is called \emph{supersingular}. This is equivalent to having a non-commutative endomorphism ring over $\overline\F_q$, see \cite[p.246]{Deu41}.
Let $N_q(\beta)$ be the number of isomorphism classes of elliptic curves over $\F_q$ in the isogeny class determined by the characteristic polynomial $h=x^2+\beta x +q$ weighted with the reciprocal of the number of automorphisms over $\F_q$.
As a consequence of Corollary \ref{cor:isomclass} we get the following Proposition, which reproves a well known result by Deuring and Waterhouse. See \cite{Wat69}, \cite{Deu41} and also \cite[Theorem 4.6]{Schoof87}.
\begin{prop}
\label{prop:ordinaryellcurves}
 Let $q=p^r$, where $p$ is a prime number and $r$ is a positive integer. Let $\beta$ be an integer satisfying $\beta^2<4q$.
 If $r>1$ assume also that $\beta$ is coprime with $p$.
 Then
 \[ N_q(\beta)=\frac{\#\Pic(\cO_K)}{\cO_K^\times}\sum_{n|f}n\prod_{p|n}\left( 1-\left( \frac{\Delta_K}{p} \right)\frac{1}{p} \right), \]
 where $K=\Q[x]/(h)$, $R=\Z[x]/(h)$ and $f=[\cO_K:R]$ with $h=x^2+\beta x +q$.
\end{prop}
\begin{proof}
 The assumptions on $\beta$ mean that, for an elliptic curve $E$ in the isogeny class determined by $\beta$, we have $E \in \AVord{q}$ or $E \in \AVcs{p}$, since there is no characteristic polynomial of an elliptic curve over $\F_p$ with root $\sqrt{p}$.
 Observe that if we write $R=\Z[\alpha]$ then $q/\alpha$ is in $R$.
 In a quadratic field every order is Bass and hence by \cite[Proposition 3.7]{MarICM18} we have
 \[ \ICM(R) = \bigsqcup_{R\subseteq S \subseteq \cO_K} \Pic(S). \]
 Therefore, by Corollary \ref{cor:isomclass} and Corollary \ref{cor:prod}.\ref{b:cor:prod}, we obtain
 \[ N_q(\beta) =\sum_{R\subseteq S \subseteq \cO_K} \dfrac{\#\Pic(S)}{\#S^\times}.\]
 Since $K$ is a quadratic field we know that each order $S$ is uniquely determined by its index $[\cO_K:S]$ and these are precisely the divisors of $f$.
 To conclude we just need to observe that if $[\cO_K:S]=n$ then we have that 
 \[\# \Pic(S) = \frac{\#\Pic(\cO_K)}{[\cO_K^\times:S^\times]}n\prod_{p|n}\left( 1-\left( \frac{\Delta_K}{p} \right)\frac{1}{p} \right),\]
 where $\left(\frac{\cdot}{p}\right)$ is the Legendre symbol for $p$ odd and the Kronecker symbol for $p=2$, see \cite[Theorem 7.24]{cox13}.
\end{proof}

\subsection*{Higher dimension}

Here we present some examples in dimension greater than~$1$.
The code to recompute them is available at \url{https://raw.githubusercontent.com/stmar89/AbVarFq/master/examples/examples_pol_sqfree_abvar.txt}.

\begin{example}
    Consider the polynomial $h=x^4 + 2x^3 - 7x^2 + 22x + 121$.
    By \cite[Theorem 1.3]{Howe95} we know that the corresponding isogeny class of simple abelian surfaces over $\F_{11}$ does not contain a principally polarized variety.
%    We now list the isomorphism classes and for each class the smallest degree $N$ for which there is at least one polarization of degree $N$, then all non-isomorphic polarizations of degree $N$ and the automorphisms of the polarized abelian variety.
    Put $K=\Q[x]/h=\Q(\alpha)$.  
%   Note that $\alpha=\zeta(1-\sqrt{10})$, where $\zeta=(-1+\sqrt{3})/2$. 
    Let $R$ be the order $\Z[\alpha,11/\alpha]$. 
    The only proper over-order of $R$ is the maximal order $\cO_K$.
    Since both orders are Gorenstein, the isomorphism classes of the abelian varieties in the isogeny class determined by $h$ functorially correspond to 
    \[ \Pic(R) \sqcup \Pic(\cO_K).\]
    Moreover we have $\Pic(R)\simeq \Z/2\Z\times \Z/2\Z$ and $\Pic(\cO_K)\simeq \Z/2\Z$, so in particular we have $6$ isomorphism classes of abelian varieties. 
    Two of the $4$ isomorphism classes with endomorphism ring $R$ have $2$ non-isomorphic polarizations of degree $4$ while the other $2$ have $2$ non-isomorphic polarizations of degree $25$.
    One of the isomorphism classes with endomorphism ring $\cO_K$ has $2$ non-isomorphic polarizations of degree $4$ while the other has $2$ non-isomorphic polarizations of degree $25$.
    The degrees mentioned above are minimal, in the sense that the isomorphism class does not admit a polarization of smaller degree. 
    All the above polarized varieties have automorphism groups of order $2$.
\end{example}

\begin{example}
 Let $h=x^6 - 2x^5 - 3x^4 + 24x^3 - 15x^2 - 50x + 125$. This is the characteristic polynomial of an isogeny class of simple abelian varieties over $\F_{5}$ of dimension $3$. 
 Put $K=\Q(\alpha)=\Q[x]/h$ and denote by $R$ the order $\Z[\alpha,5/\alpha]$.
 There are $5$ over-orders of $R$, all stable under complex conjugation.
 One of the over orders is not Gorenstein. We denote this order by $S$.
 Moreover denote by $T$ the unique over-order of $R$ such that $[\cO_K:T]=2$ and the group of torsion units is $\mu(T^\times)\simeq \Z/2\Z$.
% :
% {\scriptsize \begin{align*}
%  S_1 = & \Z \oplus  \alpha\Z \oplus  \alpha^2\Z \oplus  \alpha^3\Z \oplus \frac15(\alpha^4 + 3\alpha^3 + 2\alpha^2 + 4\alpha)\Z\oplus \frac{1}{25}(\alpha^5 + 3\alpha^4 + 12\alpha^3 + 9\alpha^2 + 5\alpha)\Z\\
%  S_2 = & \Z \oplus  \alpha\Z \oplus  \alpha^2\Z \oplus  \alpha^3\Z \oplus  \frac15(\alpha^4 + 3\alpha^3 + 2\alpha^2 + 4\alpha)\Z \oplus \frac{1}{50}(\alpha^5 + 3\alpha^4 + 12\alpha^3 + 34\alpha^2 + 5\alpha + 25)\Z\\
%  S_3 = & \Z \oplus  \alpha\Z \oplus  \alpha^2\Z \oplus  \alpha^3\Z \oplus \frac{1}{10}(\alpha^4 + 8\alpha^3 + 2\alpha^2 + 4\alpha + 5)\Z \oplus\frac{1}{50}(\alpha^5 + 3\alpha^4 + 12\alpha^3 + 34\alpha^2 + 5\alpha + 25)\Z\\
%  S_4 = & \Z \oplus  \alpha\Z \oplus  \alpha^2\Z \oplus  \alpha^3\Z \oplus  \frac{1}{5}(\alpha^4 + 3\alpha^3 + 2\alpha^2 + 4\alpha)\Z \oplus \frac{1}{100}(\alpha^5 + 3\alpha^4 + 12\alpha^3 + 84\alpha^2 + 5\alpha + 75)\Z\\
%  S_5 = & \Z \oplus  \alpha\Z \oplus  \alpha^2\Z \oplus  \alpha^3\Z \oplus \frac{1}{10}(\alpha^4 + 8\alpha^3 + 2\alpha^2 + 4\alpha + 5)\Z\oplus\frac{1}{100}(\alpha^5 + 3\alpha^4 + 12\alpha^3 + 84\alpha^2 + 5\alpha + 75)\Z
% \end{align*}}
%Observe that $S_2$ is not Gorenstein.

The $\ICM(R)$ consists of $14$ classes, so there are $14$ isomorphism classes of abelian threefolds over $\F_{5}$ with characteristic polynomial $h$.
Among these $14$ classes $2$ are not invertible in their multiplicator ring $S$.

We compute that $8$ isomorphism classes are principally polarized. 
They are all invertible in their multiplicator rings.
More precisely, they correspond to isomorphism classes in
\[ \Pic(R) \sqcup \Pic(T) \sqcup \Pic(\cO_K) \]
and all admit a unique principal polarization up to isomorphism.
The polarized isomorphism classes with endomorphism ring $R$ and $T$ have $2$ automorphisms and the one with maximal endomorphism ring have automorphism group $\Z/4\Z$.

\end{example}

\begin{example}
 Let 
 \[h=x^8 - 5x^7 + 13x^6 - 25x^5 + 44x^4 - 75x^3 + 117x^2 - 135x + 81.\]
 This is the characteristic polynomial of an isogeny class of simple abelian varieties over $\F_{3}$ of dimension $4$.
 Let $K=\Q(\alpha)=\Q[x]/h$ and denote by $R$ the order $\Z[\alpha,3/\alpha]$.
 There are $8$ over-orders of $R$.

The $\ICM(R)$ consists of $18$ classes, so there are $18$ isomorphism classes of abelian fourfolds over $\F_{3}$ with characteristic polynomial $h$.
Among these $18$ classes, $5$ are not invertible in their multiplicator rings.
It turns out that $10$ out of the $18$ ideal classes are isomorphic to the class of the conjugate of the trace dual ideal and $2$ of them are non-invertible. This means that the corresponding abelian varieties are isomorphic to their dual. 
Not all of them are principally polarized.

There are $8$ isomorphism classes which are principally polarized, all admitting a unique principal polarization up to isomorphism. The ideals corresponding to $2$ of them are not invertible in their multiplicator ring. All the principal polarized abelian varieties have automorphism group of order $2$, but the ones with maximal endomorphism ring which have $10$ automorphisms.

%\textcolor{red}{gppts}\\
 We also notice that in this example there are abelian varieties with the same endomorphism ring, but non-isomorphic groups of rational points.

\end{example}

\begin{example}
  In the following table we present the results of our computations of the isomorphism classes of all ordinary square-free isogeny classes of abelian surface over $\F_p$ for $p=2,3,5,7$ and $11$.
  We will use the following notation:
  \begin{itemize}
  \item $N_1$: ordinary square-free isogeny classes over $\F_p$,
  \item $N_2$: isomorphism classes of abelian varieties,
  \item $N_3$: isomorphism classes of abelian varieties which do not admit a principal polarization, 
  \item $N_4$: polarized isomorphism classes of principally polarized abelian varieties, 
  \item $N_5$: isomorphism classes of abelian varieties with maximal endomorphism ring,
  \item $N_6$: isomorphism classes of abelian varieties with maximal endomorphism ring which do not admit a principal polarization.
  \end{itemize}
\begin{center}
\begin{tabular}{| c | c | c | c | c | c | c |}
\hline
    $p$  & $N_1$ 	  & $N_2$ 		& $N_3$		& $N_4$	& $N_5$ 	& $N_6$ \\\hline
    $2$  & $14$            & $21$	  	& $7$ 	       	& $15$	& $15$	& $3$ \\\hline     
    $3$  & $36$            & $76$	  	& $23$ 	       	& $59$	& $43$	& $6$\\\hline
    $5$  & $94$            & $457$	  	& $203$ 	       	& $290$	& $159$	& $34$\\\hline
    $7$  & $168$          & $1324$  	& $636$ 	       	& $797$	& $387$	& $88$\\\hline
    $11$ & $352$         & $4925$  	& $2675$ 	       	& $2797$	& $1476$ & $459$\\\hline
\end{tabular}
\end{center}
  We remark that the proportion of abelian surfaces that do not admit a principal polarization is much lower when we restrict ourselves to the surfaces with endomorphism ring which is the maximal order of the endomorphism algebra.
  We have on-going computations for $g=3$ that show that this difference becomes much more pronounced.
\end{example}

\section{ Period Matrices }
\label{sec:periodmatrices}
Let $A$ be an abelian variety in $\AV(h)$ for a square-free ordinary polynomial $h$ in $\Wpoly{q}$ of degree $2g$ and $I$ be the corresponding fractional $R$-ideal, where $R=\Z[F,V]$ as usual.
Let $A'$ be the complex abelian variety $\tilde A \otimes_\varepsilon \C$ as in Remark \ref{rmk:Delfun}. Recall that $I=H_1(A',\Z)$ as abelian groups and choose a $\Z$-basis of $I$, say
\[ I=\alpha_1\Z\oplus\ldots\oplus\alpha_{2g}\Z. \]
Assume also that $A$ admits a principal polarization $\lambda$, which corresponds to multiplication by an element $a$ in $K^\times$.
Denote with $\lambda'$ the polarization induced by $\lambda$ on $A'$.
Let $\Phi=\set{\vphi_1,\ldots,\vphi_g}$ be the CM-type found by Algorithm \ref{alg:cmtype}.
Recall by \cite[Section 8]{Del69} that this particular CM-type characterizes the complex structure on $I\otimes \R$ induced by the identification with the lie algebra of the complex abelian variety $A'$, via the isomorphism of complex tori
\[ A'(\C) \simeq \dfrac{\C^g}{\Phi(I)},\]
where $\Phi(I)$ is the lattice in $\C^g$ spanned by the complex vectors
\[ (\vphi_1(\alpha_i),\ldots,\vphi_g(\alpha_i))\qquad i=1,\ldots,2g. \]
A \emph{period matrix} associated to $A'$ is a $g\times 2g$ complex matrix whose columns are the coordinates of a $\Z$-basis of the full lattice $\Phi(I)$. We are interested in a matrix that captures the Riemann form induced by the polarization $\lambda'$ of $A'$.

More precisely, as in the proof of Theorem \ref{thm:idpol} we obtain that the Riemann form associated to $a$ is given by
\[ b\colon I\times I \to \Z \quad (s,t)\mapsto \Tr(\bar{ta}s). \]
We can choose now a symplectic $\Z$-basis of $I$ with respect to the form $b$, that is,
\[ I = \gamma_1 \Z \oplus \ldots \oplus \gamma_g \Z \oplus \beta_1 \Z \oplus \ldots \oplus \beta_g \Z, \]
and
\[ b(\gamma_i,\beta_i)=1 \text{ for all $i$, and}\]
\[b(\gamma_h,\gamma_k)=b(\beta_h,\beta_k)=b(\gamma_h,\beta_k)=0 \text{ for all $h\neq k$}. \]
Such symplectic basis can be computed with appropriate modifications of the Gram-Schmidt orthogonalization process, see for example \cite[Theorem 1.1]{daSilva01}.

Consider the $g\times 2g$ matrix $\Omega$ whose $i$-th row is 
\[(\vphi_i(\gamma_1),\ldots,\vphi_i(\gamma_g),\vphi_i(\beta_1),\ldots,\vphi_i(\beta_g)).\]
This is what is usually called the \emph{big period matrix} of $(A',\lambda')$. If we write $\Omega=(\Omega_1,\Omega_2)$ we can recover the $g\times g$ \emph{small period matrix} or \emph{Riemann matrix} $\tau$ by
\[\tau= \Omega_2^{-1}\Omega_1.\]

\begin{example}
   Let $f=(x^4 - 4x^3 + 8x^2 - 12x + 9)(x^4 - 2x^3 + 2x^2 - 6x + 9)$, which identifies an isogeny class of abelian four-folds over $\F_3$.
   We compute the principally polarized abelian varieties and we find that $4$ isomorphism classes admit a unique principal polarization.
   Here we present one of them with the corresponding (approximations of the) big and small period matrices.
 {\scriptsize\begin{align*}
    \begin{split}
       I & =  \dfrac{1}{54}\left(432-549\alpha+441\alpha^2-331\alpha^3+186\alpha^4-81\alpha^5+33\alpha^6-7\alpha^7\right)\Z\oplus\\
	   & \oplus\dfrac{1}{6}\left(63-78\alpha+65\alpha^2-49\alpha^3+27\alpha^4-12\alpha^5+5\alpha^6-1\alpha^7\right)\Z\oplus\\
           & \oplus\dfrac{1}{6}\left(81-99\alpha+84\alpha^2-61\alpha^3+33\alpha^4-15\alpha^5+6\alpha^6-1\alpha^7\right)\Z\oplus\\
           & \oplus\dfrac{1}{18}\left(-63+96\alpha-86\alpha^2+68\alpha^3-39\alpha^4+18\alpha^5-8\alpha^6+2\alpha^7\right)\Z\oplus(-1)\Z\oplus\\
	   & \oplus(-\alpha)\Z\oplus(-\alpha^2)\Z\oplus\dfrac{1}{9}\left(81-96\alpha+81\alpha^2-64\alpha^3+33\alpha^4-15\alpha^5+6\alpha^6-\alpha^7\right)\Z
    \end{split}\\
    \begin{split}
       \End(I) & = \dfrac{1}{54}\left(432-549\alpha+441\alpha^2-331\alpha^3+186\alpha^4-81\alpha^5+33\alpha^6-7\alpha^7\right)\Z\oplus\\
		 & \oplus \dfrac{1}{6}\left(63-78\alpha+65\alpha^2-49\alpha^3+27\alpha^4-12\alpha^5+5\alpha^6-1\alpha^7\right)\Z\oplus\\
		 & \oplus \dfrac{1}{6}\left(81-99\alpha+84\alpha^2-61\alpha^3+33\alpha^4-15\alpha^5+6\alpha^6-1\alpha^7\right)\Z\oplus\\
		 & \oplus \dfrac{1}{18}\left(-63+96\alpha-86\alpha^2+68\alpha^3-39\alpha^4+18\alpha^5-8\alpha^6+2\alpha^7\right)\Z\oplus\\
		 & \oplus \dfrac{1}{54}\left(-378+549\alpha-441\alpha^2+331\alpha^3-186\alpha^4+81\alpha^5-33\alpha^6+7\alpha^7\right)\Z\oplus\\
		 & \oplus \dfrac{1}{6}\left(-63+84\alpha-65\alpha^2+49\alpha^3-27\alpha^4+12\alpha^5-5\alpha^6+\alpha^7\right)\Z\oplus\\
		 & \oplus \dfrac{1}{6}\left(-81+99\alpha-78\alpha^2+61\alpha^3-33\alpha^4+15\alpha^5-6\alpha^6+\alpha^7\right)\Z\oplus\\
		 & \oplus \dfrac{1}{18}\left(-99+96\alpha-76\alpha^2+60\alpha^3-27\alpha^4+12\alpha^5-4\alpha^6\right)\Z
    \end{split}\\
    x & = \dfrac{537}{80} -\dfrac{1343}{120}\alpha +\dfrac{1343}{144}\alpha^2 -\dfrac{419}{60}\alpha^3 +\dfrac{337}{80}\alpha^4 -\dfrac{15}{8}\alpha^5 +\dfrac{559}{720}\alpha^6 -\dfrac{1}{5}\alpha^7
    \end{align*}

    \begin{align*}
    & \Omega =
    \begin{pmatrix}
      2.8 - i & -2.8 + 0.59i & 0 & 0 & 1 & 1.7 - 0.29i & 0 & 0 \\
      -2.8 + i & 2.8 - 3.4i & 0 & 0 & 1 & 0.29 + 1.7i & 0 & 0 \\
      0 & 0 & -1 & -0.38 - 0.15i & 0 & 0 & -1.6 - 0.62i & -0.15 - 0.15i\\
      0 & 0 & -1 & -2.6 + 6.9i & 0 & 0 & 0.62 - 1.6i & -6.9 + 6.9i
    \end{pmatrix},\\
    & \tau =
    \begin{pmatrix}
    -1 - 2.8i & 2 + 1.4i & 0 & 0 \\
    2 + 1.4i & -2.7 - 0.95i & 0 & 0 \\
    0 & 0 & 0.52 - 0.21i & 0.14 \\
    0 & 0 & 0.14 & 0.71 - 0.31i
    \end{pmatrix}
    \end{align*}}  
\end{example}

%%%%%%%%%%%%%%%%%%%%%%%%%%%%%%%%%%%%%%%%%%%%%%%%%%%%%%%%%%%%%%%%%%%%%%%%%%%%%%%%%%%%%%%%%%%%%%%%%%%%%%%%%%%%%%%%%%%%%%
\bibliographystyle{amsalpha}
%\renewcommand{\bibname}{References} % changes the header from Bibliography to References
%\bibliography{references} % adjust this to fit your BibTex file

\newcommand{\etalchar}[1]{$^{#1}$}
\def\cprime{$'$}
\providecommand{\bysame}{\leavevmode\hbox to3em{\hrulefill}\thinspace}
\providecommand{\MR}{\relax\ifhmode\unskip\space\fi MR }
\providecommand{\MRhref}[2]{%
  \href{http://www.ams.org/mathscinet-getitem?mr=#1}{#2}
}
\providecommand{\href}[2]{#2}

\end{document}